\documentclass[leqno,12pt]{amsart} 
\usepackage{amssymb}
\usepackage[pdftex]{graphicx}
\theoremstyle{plain} 
\newtheorem{theorem}{\indent\sc Theorem}[section]

\newtheorem{corollary}[theorem]{\indent\sc Corollary}
\newtheorem{proposition}[theorem]{\indent\sc Proposition}
\newtheorem{fact}[theorem]{\indent\sc Fact}

\newtheorem{problem}[theorem]{\indent\sc Problem}
\newtheorem{conjecture}[theorem]{\indent\sc Conjecture}
\theoremstyle{definition} 
\newtheorem{definition}[theorem]{\indent\sc Definition}
\newtheorem{remark}[theorem]{\indent\sc Remark}
\newtheorem{example}[theorem]{\indent\sc Example}

\def\R{{\mathbf{R}}}
\def\Pi{{\mathbf{P}}}

\begin{document}

\title[Gauss images of complete minimal surfaces of genus zero]{The Gauss images of complete minimal surfaces of genus zero of finite total curvature} 

\author[Y.~Kawakami]{Yu Kawakami} 

\author[M.~Watanabe]{Mototsugu Watanabe} 

\dedicatory{Dedicated to Professors Masaaki Umehara and Kotaro Yamada on their sixtieth birthdays}

\renewcommand{\thefootnote}{\fnsymbol{footnote}}
\footnote[0]{2020\textit{ Mathematics Subject Classification}.
Primary 53A10; Secondary 30D35, 53A05.}
\thanks{
This work was supported by JSPS KAKENHI Grant Number JP19K03463, JP23K03086. 
}

\address{
Faculty of Mathematics and Physics, Kanazawa University \endgraf
Kanazawa, 920-1192, Japan
}
\email{y-kwkami@se.kanazawa-u.ac.jp}

\address{
Toyama Prefectural Chuo Agricultural High School, \endgraf
Toyama, 930-1281, Japan
}
\email{mototsuguw03@gmail.com}


\maketitle

\begin{abstract}
This paper aims to present a systematic study on the Gauss images of complete minimal surfaces of genus $0$ of finite total curvature in Euclidean $3$-space and Euclidean $4$-space. We focus on the number of omitted values and the total weight of the totally ramified values of their Gauss maps. In particular, we construct new complete minimal surfaces of finite total curvature whose Gauss maps have $2$ omitted values and $1$ totally ramified value of order $2$, that is, the total weight of the totally ramified values of their Gauss maps are $5/2\,(=2.5)$ in Euclidean $3$-space and Euclidean $4$-space, respectively. Moreover we discuss several outstanding problems in this study.   
\end{abstract}

\section{Introduction}\label{sec-Int} 

The study of the Gauss images of complete minimal surfaces in Euclidean space have been of interest ever since the Nirenberg conjecture and the Osserman work (see \cite[Chapter 8]{Os1986}). One of the most 
notable results in this study is the Fujimoto theorem \cite{Fu1988}, which states that the Gauss map of a nonflat complete minimal surface in Euclidean $3$-space $\mathbf{R}^3$ can omit at most $4$ values in the Riemann sphere 
$\overline{\mathbf{C}} := \mathbf{C}\cup \{ \infty \} \simeq S^{2}$. This theorem is optimal because the Gauss map of the classical Scherk surface omits exactly $4$ values in $\overline{\mathbf{C}}$.  On the other hand, Osserman \cite{Os1964} showed that the Gauss map of a nonflat complete minimal surface of finite total curvature in $\mathbf{R}^3$ can omit at most $3$ values in $\overline{\mathbf{C}}$. However, no complete minimal surface of finite total curvature in $\mathbf{R}^3$ whose Gauss map omits $3$ distinct values in $\overline{\mathbf{C}}$ has been found. So the precise maximum is expected to be ``$2$''. There have been several attempts (e.g. \cite{KKM2008, MJ2018}), but no complete solution has been reached.        
   
On the other hand, in the theory of value distribution of meromorphic functions, totally ramified values are studied as well as omitted values. Here we recall the notion of an omitted value and a totally ramified value of a meromorphic function.  

\begin{definition}\label{def-total}
Let $f\colon \Sigma \to \overline{\mathbf{C}}$ be a meromorphic function on a Riemann surface $\Sigma$. 
We call a value that $f$ never attains an {\it omitted value} of $f$. We denote the number of omitted values of $f$ by $D_{f}$. 
Set $\nu\, (\geq 2) \in \mathbf{Z}\cup \{ \infty \}$. We say that $\alpha \in \overline{\mathbf{C}}$ is a {\it totally ramified value} of $f$ of {\it order} $\nu$ if the equation $f=\alpha$ has no root of multiplicity less than $\nu$ on $\Sigma$. We regard an omitted value of $f$ as a totally ramified value of $f$ of order $\infty$ because $\nu =\infty$ means that $f=\alpha$ has no root of any order. Then the {\it weight} for a totally ramified value of $f$ of order $\nu$ is defined by $1-(1/\nu)$. By {\it total weight} $\nu_{f}$ of the totally ramified values of $f$, we mean the sum of their weights. By definition, we have $D_{f}\leq \nu_{f}$.   
\end{definition}
 
In order to explain the natural interpretation for the total weight of the totally ramified values, we need the second main theorem in Nevanlinna theory. See \cite{Ko2003, Ne1970, NW2013, Ro1939, Ru2021} for more details. 

\begin{remark}\label{rmk-total}
The idea of the total weight of the totally ramified values is due to R. Nevanlinna 
(see \cite[Page 102]{Ne1929}, \cite[\S X.3]{Ne1970}). See \cite[Page 83]{Be2006} for historical 
background and  relationship to the theory of normal family. 
In \cite{Ka2006, Ka2009, KKM2008}, the value $\nu_{f}$ in Definition \ref{def-total} is called `` the totally ramified value number of $f$''. However, this term is confusing with ``the number of totally ramified values of $f$'', so we refer $\nu_{f}$ as ``the total weight of the totally ramified values of $f$'' with reference to the Robinson paper \cite{Ro1939}. On the other hand, in \cite{Ro1939}, ``totally ramified value'' in Definition \ref{def-total} is called ``exceptional value''. However, ``exceptional value'' is often used in the sense of ``omitted value'', so  we choose  the term ``totally ramified value'' (for example, this term is used in the Nevanlinna book \cite{Ne1970}) in this paper.  
\end{remark}

By Definition \ref{def-total}, the total weight of the totally ramified values gives more detailed information than the number of omitted values. We demonstrate it for nonconstant meromorphic functions 
on the complex plane $\mathbf{C}$. By the little Picard theorem, a nonconstant meromorphic function $f\colon \mathbf{C}\to \overline{\mathbf{C}}$ can omit at most $2$ values, that is, $D_{f}\leq 2$. On the other hand, by using the fundamental theorem of algebra, we can show that a nonconstant rational function $f\colon \mathbf{C}\to \overline{\mathbf{C}}$ can omit at most $1$ value, that is, $D_{f}\leq 1$. Thus the difference between the general case and the algebraic case seems significant regarding the maximum for the number of omitted values. 
However, regarding the maximum for the total weight of the totally ramified values, the difference becomes smaller. In fact, we have $\nu_{f}\leq 2$ (Fact \ref{thm-rami1}) for a nonconstant meromorphic function $f\colon \mathbf{C} \to \overline{\mathbf{C}}$, but we can show $\nu_{f}\leq 2-(1/d)$ for a nonconstant rational function $f\colon \mathbf{C} \to \overline{\mathbf{C}}$ whose degree is $d$ (Proposition \ref{prop-rational}).   

The purpose of this paper is to perform a systematic study on the Gauss images of complete minimal surfaces in $\mathbf{R}^3$ and $\mathbf{R}^4$. We focus on the number of omitted values and the total weight of the totally ramified values of their Gauss maps. The paper is organized as follows: In Section \ref{sec-mero}, we review the ramification estimates (Fact \ref{thm-rami1} and Proposition \ref{prop-rational}) for nonconstant meromorphic functions on $\mathbf{C}$ in detail. In Section \ref{sec-R3}, we study the number of omitted values and the total weight of the totally ramified values of the Gauss map of complete minimal surfaces in $\mathbf{R}^3$. In particular, we give a new complete minimal surface of finite total curvature in $\mathbf{R}^3$ whose Gauss map has $2$ omitted values and $1$ totally ramified value of order $2$, that is, the total weight of the totally ramified value of its Gauss map is $5/2\,(=2.5)$ (Theorem \ref{M-thm-1}). Moreover we obtain more precise estimate for the total weight of the totally ramified values of the Gauss map of complete minimal surfaces of genus $0$ of finite total curvature in $\mathbf{R}^3$ (Theorem \ref{thm-rami3}). In Section \ref{sec-R4}, we also study the number of omitted values and the total weight of the totally ramified values of the Gauss map of complete minimal surfaces in $\mathbf{R}^4$. In particular, we construct a complete minimal surface of finite total curvature in $\mathbf{R}^4$ whose both Gauss maps have $2$ omitted values and $1$ totally ramified value of order $2$, that is, the total weight of the totally ramified values of each Gauss map is $5/2$ (Theorem \ref{thm-R4-HKW-1}). Moreover we obtain more precise estimates for the number of omitted values (Corollary \ref{thm-R4-M3}) and the total weight of the totally ramified values (Theorem \ref{thm-R4-M2}) of the Gauss map of complete minimal surfaces of genus $0$ of finite total curvature in $\mathbf{R}^4$. Furthermore we provide examples (Theorems \ref{pro-R4-M4}, \ref{pro-R4-M5}) which show that its estimate for the number of omitted values  (Corollary \ref{thm-R4-M3}) is sharp. In Section \ref{sec-Prob}, we explain outstanding problem and conjecture in this study. 

Finally, the authors gratefully acknowledge the useful comments from Shunsuke Kasao and warm encouragement from Shoichi Fujimori, Pham Hoang Ha and Reiko Miyaoka during the preparation of this paper. The authors also would like to express their gratitude to the reviewer for valuable comments and suggestions.   

\section{The images of meromorphic functions on $\mathbf{C}$}\label{sec-mero} 

In this section, we discuss the ramification estimates for meromorphic functions on $\mathbf{C}$. By an application of the defect relation 
 (see \cite{Ko2003, Ne1970, NW2013, Ru2021} for instance) in Nevanlinna theory for meromorphic functions 
on $\mathbf{C}$, the following ramification estimate is given. 

\begin{fact}\label{thm-rami1}
Consider a nonconstant meromorphic function $f\colon \mathbf{C}\to \overline{\mathbf{C}}$. 
Let $D_{f}$ be the number of omitted values and $\nu_{f}$ be the total weight of the totally ramified values of $f$. Then we have 
\begin{equation}\label{eq-rami1}
D_{f} \leq \nu_{f} \leq 2. 
\end{equation}
\end{fact}

The value ``$2$'' in \eqref{eq-rami1} is optimal because $f(z)=e^z$ has exactly $2$ omitted values $0, \infty$, that is, $D_{f}=2$. Moreover the Weierstrass $\wp$-function $f(z) =\wp (z)$ with the period $\omega_{1}, \omega_{2}$ has just $4$ totally ramified values of order $2$, $e_{1}=\wp (\omega_{1}/2),\, e_{2}=\wp (\omega_{2}/2),\, e_{3}=\wp ((\omega_{1}+\omega_{2})/2),\, \infty$. We thus obtain $\nu_{f}= 4(1-(1/2))=2$. The geometric interpretation for ``$2$'' is the Euler characteristic of the Riemann sphere (See \cite{Ah1935, Ch1960} for instance). 

On the other hand, by the argument of complex analytic geometry, the following result holds for nonconstant rational functions on $\mathbf{C}$.  

\begin{proposition}\label{prop-rational}
Let $f\colon \mathbf{C} \to \overline{\mathbf{C}}$ be a nonconstant rational function whose degree is $d\, (\geq 1)$. Then the number $D_{f}$ of omitted values and the total weight $\nu_{f}$ of the totally ramified values of $f$ satisfy 
\begin{equation}\label{eq-rational}
D_{f} \leq \nu_{f} \leq 2-\dfrac{1}{d}. 
\end{equation}
In particular, we have $\nu_{f}< 2$ and $D_{f}\leq 1$.  
\end{proposition}

\begin{proof}
Let $\# (A)$ be the cardinality of the set $A$ and $e_{p}(f)$ be the multiplicity of $f$ at $p$. 
Since $f$ is a rational function, $f$ can be extended meromorphically to $\overline{\mathbf{C}}$. 

Let $a_{1}, \ldots, a_{D_{f}}$ be the omitted values of $f\colon \mathbf{C} \to \overline{\mathbf{C}}$. 
Set
\[
n_{0} := \displaystyle \sum_{i=1}^{D_{f}} \left( \sum_{p\in f^{-1}(a_{i})} (e_{p}(f) -1) \right).  
\]
Then we obtain 
\[
dD_{f} = \displaystyle \sum_{i=1}^{D_{f}} \left( \sum_{p\in f^{-1}(a_{i})} e_{p}(f) \right) = n_{0} + \# \left( \bigcup_{i=1}^{D_{f}} f^{-1} (a_{i}) \right). 
\]
On the other hand, we have 
\begin{equation}\label{eq-rational0}
\bigcup_{i=1}^{D_{f}} f^{-1} (a_{i}) \subset \{ \infty \}
\end{equation}
because $f\colon \overline{\mathbf{C}} \to \overline{\mathbf{C}}$ is surjective. Thus we obtain 
\begin{equation}\label{eq-rational2}
dD_{f} \leq n_{0} + \# (\{ \infty \}) = n_{0} +1.  
\end{equation}
This inequality obviously holds for $D_{f}=0$. By \eqref{eq-rational0}, we have $D_{f}\leq 1$.  

Let $b_{1}, \ldots, b_{l_{0}}$ be the totally ramified values which are not omitted values. 
Set 
\[
n_{r} := \displaystyle \sum_{j=1}^{l_{0}} \left( \sum_{p\in f^{-1}(b_{j})} (e_{p}(f) -1) \right).  
\] 
For each $b_{j}$, the order $\nu_{j}$ is equal to the minimum of the multiplicity at all $f^{-1}(b_{j})$. 
Then the number of elements in the inverse image $f^{-1}(b_{j})$ is less than or equal to $d/\nu_{j}$. We thus obtain 
\begin{equation}\label{eq-rational3}
\displaystyle dl_{0} = \sum_{j=1}^{l_{0}} \left( \sum_{p\in f^{-1}(b_{j})} e_{p}(f) \right) = n_{r} + \# \left( \bigcup_{j=1}^{l_{0}} f^{-1} (b_{j}) \right) \leq n_{r} + \sum_{j=1}^{l_{0}} \dfrac{d}{\nu_{j}}. 
\end{equation}

Let $n_{f}$ be the total branching order of $f$. In particular, we have $n_{0}+n_{r}\leq n_{f}$. 
Then applying the Riemann-Hurwitz theorem to $f\colon \overline{\mathbf{C}} \to \overline{\mathbf{C}}$, we obtain $n_{f}=2(d-1)$. By \eqref{eq-rational2} and \eqref{eq-rational3}, we have 
\[
\displaystyle \nu_{f} = D_{f} +\sum_{j=1}^{l_{0}} \left( 1-\dfrac{1}{\nu_{j}} \right) \leq \dfrac{n_{0}+1}{d}+\dfrac{n_{r}}{d}\leq \dfrac{n_{f}+1}{d} = 2-\dfrac{1}{d}.  
\]
\end{proof}

The value ``$2-(1/d)$'' in \eqref{eq-rational} is optimal because $f(z)=z^{d}$ has $1$ omitted value $\infty$ and 
$1$ totally ramified value $0\, (=f(0))$ of order $d$. We thus obtain $\nu_{f} =1+(1-(1/d)) =2-(1/d)$. We note that 
Jin and Ru show the algebraic Nevanlinna second main theorem (see \cite[Theorem 2.1]{JR2007} and \cite[Theorem A1.1.3]{Ru2023}) by using the argument 
in the proof of Proposition \ref{prop-rational}.  

\begin{remark}\label{rmk-puncture}
If a meromorphic function $f\colon \mathbf{C}\backslash \{ 0 \} \to \overline{\mathbf{C}}$ can be extended meromorphically to $\overline{\mathbf{C}}$, 
by the same argument of the proof of Proposition \ref{prop-rational}, then we have $dD_{f}\leq n_{0}+2$ and 
\[
\nu_{f} = D_{f} +\sum_{j=1}^{l_{0}} \left( 1-\dfrac{1}{\nu_{j}} \right) \leq \dfrac{n_{0}+2}{d}+\dfrac{n_{r}}{d}\leq \dfrac{n_{f}+2}{d} = 2.   
\]
\end{remark}


\section{The Gauss images of complete minimal surfaces \\ of finite total curvature in $\mathbf{R}^3$}\label{sec-R3}

We first recall fundamental results of complete minimal surfaces in $\mathbf{R}^3$. 
Details can be found, for example, \cite{AFL2021, Fu1993, Ka2020, Ko2021, Lo1980, Os1986, Ru2023, Ya1994}.  
Let $X = (x^1, x^2, x^3)\colon \Sigma \to \mathbf{R}^3$ be an oriented minimal surface. 
By associating a local complex coordinate $z=u+\mathrm{i}v$ ($\mathrm{i}:= \sqrt{-1}$) 
with each positive isothermal coordinate $(u, v)$, $\Sigma$ is considered 
as a Riemann surface whose conformal 
metric is the induced metric $ds^2$ from $\mathbf{R}^3$. Then
\begin{equation}\label{eq-Lap}
\Delta_{ds^2} X=0
\end{equation}     
holds, that is, each coordinate function $x^{i}\, (i=1, 2, 3)$ is harmonic. With respect to the local complex 
coordinate $z=u+\mathrm{i}v$ of the surface, \eqref{eq-Lap} is given by 
\begin{equation}\label{eq-partial}
\bar{\partial} \partial X =0, 
\end{equation}
where $\partial = (\partial /\partial u - \mathrm{i}\partial /\partial v)/2, \bar{\partial} = (\partial /\partial u + \mathrm{i}\partial /\partial v)/2$. 
Hence each $\phi_{i}:= \partial x^{i}\,dz\, (i=1, 2, 3)$ 
is a holomorphic differential on $\Sigma$. These satisfy 
\begin{itemize}
\item $\sum_{i} {\phi_{i}}^{2} = 0$: conformality condition, 
\item $\sum_{i} |\phi_{i}|^{2} >0$: regularity condition, 
\item Each $\phi_{i}$ has no real periods on $\Sigma$, that is, $\mathrm{Re}\int_{c} \phi_{i}= 0$ holds 
for every cycle $c \in H_{1} (\Sigma, \mathbf{Z})$: period condition. 
\end{itemize}
We recover $X\colon \Sigma \to \mathbf{R}^3$ by 
\begin{equation}\label{eq-integral}
X(z) =\displaystyle \mathrm{Re} \left( \int^{z}_{z_{0}} 2 \phi_{1}, \int^{z}_{z_{0}} 2 \phi_{2}, \int^{z}_{z_{0}} 2 \phi_{3} \right)
\end{equation}
up to translation. Here $z_{0}$ is a fixed point of $\Sigma$. 
If we set 
\begin{equation}\label{eq-W-data}
g= \dfrac{\phi_{3}}{\phi_{1}-\mathrm{i}\phi_{2}}, \quad hdz = \phi_{1} -\mathrm{i}\phi_{2}, 
\end{equation}
then $g$ is a meromorphic function on $\Sigma$ and $hdz$ is a holomorphic differential on $\Sigma$. 
Moreover $g$ coincides with the composition of the Gauss map of $X(\Sigma)$ and the stereographic projection 
from the $2$-sphere $\mathbf{S}^2$ onto the Riemann sphere $\overline{\mathbf{C}}$. For the meromorphic function $g$ and the holomorphic differential $hdz$ given by \eqref{eq-W-data}, 
we have  
\begin{equation}\label{eq-phi}
\phi_{1}= \dfrac{1}{2}(1-g^2)\,hdz, \quad \phi_{2}=\dfrac{\mathrm{i}}{2}(1+g^2)\,hdz, \quad 
\phi_{3}= g hdz. 
\end{equation}
We call $(g, hdz)$ the {\it Weierstrass data} ({\it W-data}, for short). Conversely, if we are given the W-data $(g, hdz)$ 
on $\Sigma$, we obtain $\phi_{1}, \phi_{2}, \phi_{3}$ by \eqref{eq-phi}. 
They satisfy the conformality condition automatically, and the regularity condition is interpreted as the poles of $g$ of order $s$ coinciding exactly with the zeros of $hdz$ of order $2s$, 
that is, $(hdz)_{0} = 2(g)_{\infty}$, where $(hdz)_{0}$ is the zero divisor of $hdz$ and $(g)_{\infty}$ is the polar divisor of $g$.  
This is because the induced metric on $\Sigma$ is given by 
\begin{equation}\label{eq-metric}
ds^{2} = (1+|g|^2)^{2} |h|^2 |dz|^2. 
\end{equation}
In general, for a given meromorphic function $g$ on $\Sigma$, it is not so hard to find a holomorphic 
differential $hdz$ satisfying the regularity condition. However, the period condition always causes trouble. 
In addition, a minimal surface in $\mathbf{R}^3$ is said to be {\it complete} if all divergent paths have infinite length with respect to the metric given by \eqref{eq-metric}.  

The Gauss curvature of $X(\Sigma)$ is given by 
\begin{equation}\label{eq-G-curvature}
K = -\dfrac{4|g'|^2}{(1+|g|^2)^{4}|h|^2}, 
\end{equation}
where $g'=dg/dz$. Moreover the total curvature of $X(\Sigma)$ is given by 
\begin{equation}\label{eq-t-curvature}
C (\Sigma) := \displaystyle \int_{\Sigma} K\,dA 
= -\int_{\Sigma} \left( \dfrac{2|g'|}{1+|g|^2} \right)^2 du\wedge dv, \quad z=u+\mathrm{i}v, 
\end{equation}     
where $dA$ is the area element with respect to the metric \eqref{eq-metric}. Note that $|C (\Sigma)|$ is the area 
of $\Sigma$ with respect to the metric induced from the Fubini-Study metric of $\overline{\mathbf{C}}$ by $g$. 

Fujimoto \cite{Fu1988, Fu1992} proved the following precise estimate for the number of omitted values and the total weight of the totally ramified values of the Gauss map of complete minimal surfaces in $\mathbf{R}^3$. 
\begin{fact}\label{thm-rami2}
Consider a nonflat complete minimal surface $X\colon \Sigma \to \mathbf{R}^3$ and its Gauss map $g\colon \Sigma \to \overline{\mathbf{C}}$. Then the number $D_{g}$ of omitted 
values and the total weight $\nu_{g}$ of the totally ramified values of $g$ satisfy 
\begin{equation}\label{eq-rami2}
D_{g} \leq \nu_{g} \leq 4. 
\end{equation}
\end{fact}
This result is proved by using complex analytic methods. The value ``$4$'' in \eqref{eq-rami2} is optimal. The most famous example of a complete minimal surface in $\mathbf{R}^3$ whose Gauss map omits exactly $4$ values 
is the classical Scherk surface. Another important example is the Voss surface \cite[Theorem 8.3]{Os1986}. The W-data of this surface is defined on $\Sigma = \mathbf{C}\backslash \{ \alpha_{1}, \alpha_{2}, \alpha_{3} \}$ for 
distinct $\alpha_{1}, \alpha_{2}, \alpha_{3}\in \mathbf{C}$, by 
\[ 
(g(z), hdz) = \left( z,\, \dfrac{dz}{\prod_{j=1}^{3} (z-\alpha_{j})} \right). 
\]
Since this data does not satisfy the period condition, we obtain a complete minimal surface $X\colon \widetilde{\Sigma}\to \mathbf{R}^3$ on the universal covering surface 
$\widetilde{\Sigma}$ of $\Sigma$. 
In particular, it has infinite total curvature. In \cite{Ka2013, Ka2015}, the first author showed a geometric interpretation for the maximal number ``$4$''. 

We next give some facts about the Gauss map of complete minimal surfaces of finite total curvature in $\mathbf{R}^3$ (\cite{Hu1957, Os1964}) .  

\begin{fact}[Huber-Osserman]\label{thm-HO}
A complete minimal surface $X\colon \Sigma \to \mathbf{R}^3$ of finite total curvature 
satisfies: 
\begin{enumerate}
\item[$\rm(\hspace{.18em}i\hspace{.18em})$] $\Sigma$ is conformally equivalent to $\overline{\Sigma}_{G}\backslash \{ p_{1}, \ldots, p_{k} \}$, where $\overline{\Sigma}_{G}$ is a closed Riemann surface of genus $G$ and $p_{1}, \ldots, p_{k} \in \overline{\Sigma}_{G}$. Then we also call the number $G$ the genus of $X(\Sigma)$.     
\item[$\rm(\hspace{.18em}ii\hspace{.18em})$] The Weierstrass data $(g, hdz)$ can be extended meromorphically to $\overline{\Sigma}_{G}$. 
In particular, if $d$ is the degree of $g\colon \overline{\Sigma}_{G} \to \overline{\mathbf{C}}$, then $C(\Sigma) =-4\pi d$ holds.  
\end{enumerate}
\end{fact}

By this fact, we also call such a surface an {\it algebraic minimal surface} (\cite{Ya1994}). 
Osserman proved the following result for the number of omitted values of the Gauss map 
of complete minimal surfaces of finite total curvature in $\mathbf{R}^3$. 

\begin{fact}\cite[Theorem 3.3]{Os1964}\label{fact-Oss}
The Gauss map of a nonflat complete minimal surface of finite total curvature in $\mathbf{R}^3$ 
can omit at most 3 distinct values.    
\end{fact}

The first author, Kobayashi and Miyaoka refined Fact \ref{fact-Oss} and 
obtained the following estimate by geometric quantities for the number of omitted values and the total weight 
of the totally ramified values of the Gauss map of complete minimal surfaces of finite total curvature 
in $\mathbf{R}^3$. 

\begin{theorem}\cite[Theorem 3.3]{KKM2008}\label{thm-KKM} 
Let $X\colon \Sigma = \overline{\Sigma}_{G}\backslash \{ p_{1}, \ldots, p_{k} \} \to \mathbf{R}^3$ be a nonflat 
complete minimal surface of finite total curvature, $g\colon \Sigma \to \overline{\mathbf{C}}$ be its Gauss map and 
$d$ be the degree of $g\colon \overline{\Sigma}_{G} \to \overline{\mathbf{C}}$. Then the number $D_{g}$ of omitted 
values and the total weight $\nu_{g}$ of the totally ramified values of $g$ satisfy 
\begin{equation}\label{eq-KKM}
D_{g}\leq \nu_{g} \leq 2+\dfrac{2}{R}, \quad \dfrac{1}{R}= \dfrac{G-1+(k/2)}{d},   
\end{equation}
and $1/R <1$ holds. Thus we have 
\begin{equation}\label{eq-KKM2}
D_{g}\leq \nu_{g}< 4. 
\end{equation}   
\end{theorem}

Theorem \ref{thm-KKM} is proved by using the methods of complex analytic geometry. 
In \cite[Page 90]{Os1986}, Osserman proposed the following problem.  

\begin{problem}\label{Pro-Oss}
Does there exist a complete minimal surface of finite total curvature in $\mathbf{R}^3$ 
whose Gauss map omits $3$ values?
\end{problem}

If so, the value ``$3$'' is the precise bound. If not, the maximum is ``$2$'' and is attained by the catenoid.  
Since no complete minimal surface of finite total curvature in $\mathbf{R}^3$ whose Gauss map omits $3$ values has been found (see \cite{Fa1993, WX1987}), the maximum is expected to be ``$2$''. 
If the best upper bound for the number of omitted values is ``$2$'', then of course the best upper bound for the total weight of the totally ramified values is also assumed to be ``$2$''.  However, the first author found a complete minimal surface of finite total curvature whose Gauss map has $2$ omitted values and $1$ totally ramified value of order $2$, that is, the total weight of the totally ramified values of its Gauss map is ``$5/2\,(=2.5)$'' in one of examples in \cite{MS1994}.   

\begin{theorem}\cite[Proposition 3.1]{MS1994}, \cite{Ka2006}\label{M-thm-0}
Consider $\Sigma =\overline{\mathbf{C}} \setminus \{\pm \mathrm{i}, \infty \}$ and 
\begin{equation}\label{MS1}
\left\{ \,
\begin{aligned}
& g(z) = \sigma \dfrac{z^2+1+a(t-1)}{z^2+t}, \\
& hdz = \dfrac{(z^2 +t)^2}{(z^2 +1)^2}\, dz, \quad (a-1)(t-1)\not= 0,   \\
& \sigma^{2} = \dfrac{t+3}{a \{ (t-1)a +4\}}.
\end{aligned}
\right.
\end{equation}
For any $a, t\in \mathbf{R}\setminus\{ 1 \}$ satisfying $\sigma^{2}< 0\,\, \text{i.e.}\,\, \sigma \in \mathrm{i}\mathbf{R}$, 
we obtain a complete minimal surface in $\mathbf{R}^3$ whose Weierstrass data is given by \eqref{MS1}. 
In particular, the surface has finite total curvature and its Gauss map has 2 omitted values $\sigma\,(=g(\infty)), \sigma a\,(=g(\pm \mathrm{i}))$ and $1$ totally ramified value $\sigma (1+a(t-1))/t\, (=g(0))$ of order 2. 
Thus the total weight $\nu_{g}$ of the totally ramified values of $g$ is $1+1+(1 - 1/2) =5/2$. 
\end{theorem}

Theorem \ref{M-thm-0} shows that \eqref{eq-KKM} in Theorem \ref{thm-KKM} is sharp in the case $(G, k, d)=(0, 3, 2)$ because we have $\nu_{g}= 5/2$ and $2+(2/R) = 2+ \{2(0-1+(3/2))/2\} =5/2$.  

No other complete minimal surfaces of finite total curvature with $\nu_{g}= 5/2$ have been found before. However, the second author obtains a new complete minimal surface of finite total curvature with $\nu_{g}=5/2$. This is one of the main results in this paper. 

\begin{theorem}\label{M-thm-1}
Consider $\Sigma =\overline{\mathbf{C}} \setminus \{0,  \pm \mathrm{i}, \infty \}$ and 
\begin{equation}\label{moto1}
\left\{ \,
\begin{aligned}
& g(z) = \sigma \dfrac{(b-a)z^4 +4a(b-1)z^2 +4a(b-1)}{(b-a)z^4 +4(b-1)z^2 +4(b-1)}, \\
& hdz = \dfrac{\{(b-a)z^4 +4(b-1)z^2 +4(b-1)\}^{2}}{z^2(z^2+1)^{2}}\, dz,  \\
& \sigma^{2} = \dfrac{5a+11b-16}{16ab-11a-5b}.  
\end{aligned}
\right. 
\end{equation}
For any $a, b \in \mathbf{R}\setminus\{ 1 \}$ satisfying $a\not= b$ and $\sigma^{2}< 0 \,\, \text{i.e.}\,\, \sigma \in \mathrm{i}\mathbf{R}$, we obtain a complete minimal surface $X(\Sigma)\colon \Sigma \to \mathbf{R}^3$ whose Weierstrass data is given by \eqref{moto1}. 
In particular, $X(\Sigma)$ has finite total curvature and $\nu_{g}= 5/2$. 
\end{theorem}

\begin{figure}[tb]
\begin{center}
\includegraphics[scale=0.08]{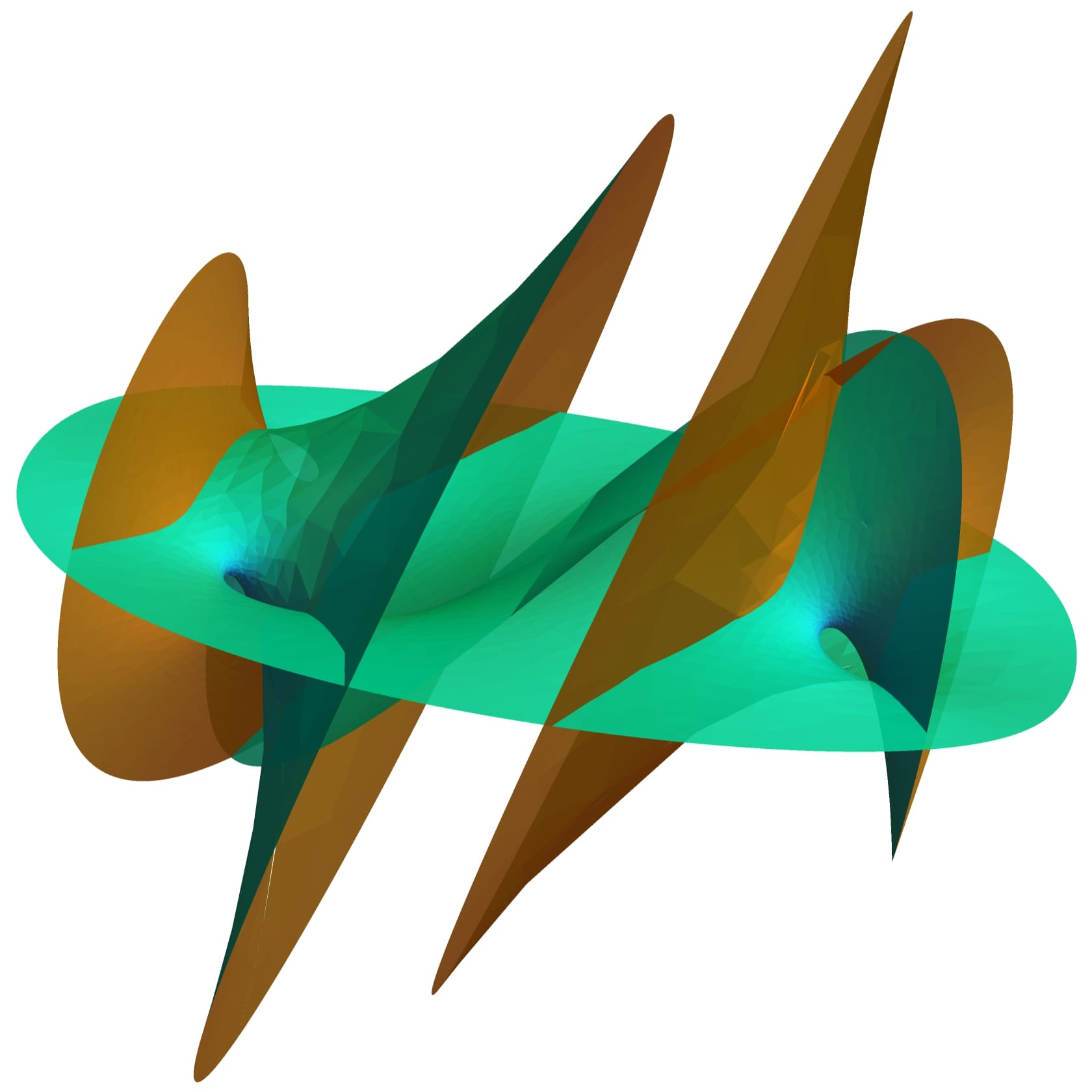}
\end{center}
\caption{Computer graphic of a minimal surface when $(a, b)=(0, 2)$ in Theorem \ref{M-thm-1}. 
This figure is drawn by Shoichi Fujimori.}
\label{fig:moto}
\end{figure}

\begin{proof}
We first prove that a minimal surface $X\colon \Sigma \to \mathbf{R}^3$ whose W-data is given by \eqref{moto1} can be constructed. We can easily verify that \eqref{moto1} satisfies the conformality and regularity conditions. So it is sufficient to prove that \eqref{moto1} satisfies the period condition.  
We denote small circles about $z=0, \mathrm{i}, -\mathrm{i}$ by $c_{0}, c_{\mathrm{i}}, c_{-\mathrm{i}}$, respectively. 
Then we obtain 
\[
\begin{aligned}
& \int_{c_{0}} \phi_{1} = \int_{c_{0}} \phi_{2} = \int_{c_{0}} \phi_{3} =0, \quad \int_{c_{\mathrm{i}}} \phi_{1} = - \int_{c_{-\mathrm{i}}} \phi_{1} = 0, \\
& \int_{c_{\mathrm{i}}} \phi_{2} = - \int_{c_{-\mathrm{i}}} \phi_{2} = -\dfrac{\pi}{2}\mathrm{i} (b-a)(5a+11b-16) \in \mathrm{i}\mathbf{R}, \\
& \int_{c_{\mathrm{i}}} \phi_{3} = - \int_{c_{-\mathrm{i}}} \phi_{3} = -\dfrac{\pi}{2}\sigma(b-a)(8ab-3a+3b-8) \in \mathrm{i}\mathbf{R}.   
\end{aligned}
\]
Since $H_{1}(\Sigma, \mathbf{Z}) =\mathbf{Z} \{ c_{0}, c_{\mathrm{i}}, c_{-\mathrm{i}} \}$, each 
$\phi_{i}\,(i=1, 2, 3)$ has no real periods on $\Sigma$, that is, \eqref{moto1} satisfies the period condition.  

We next prove that the surface $X (\Sigma)$ is complete. Since the line element $ds:=  |h|(1+|g|^2) |dz|$ of $X(\Sigma)$ 
can be represented as   
\[
\dfrac{|(b-a)z^4 +4(b-1)z^2 +4(b-1)|^2}{|z|^2 |z-\mathrm{i}|^2 |z+\mathrm{i}|^2} \left( 1+|\sigma|^2 \dfrac{|(b-a)z^4 +4a(b-1)z^2 +4a(b-1)|}{|(b-a)z^4 +4(b-1)z^2 +4(b-1)|} \right) |dz|, 
\]
we see that $\int_{\gamma} ds = +\infty$ for any divergent path $\gamma$ on $\Sigma$. 
We conclude that $X(\Sigma)$ is a complete minimal surface of finite total curvature in 
$\mathbf{R}^3$ and has $C(\Sigma) = -16 \pi$. 

Finally, we show $\nu_{g}=5/2$. The Gauss map $g$ has $2$ omitted values $\sigma\, (=g(\pm \mathrm{i}) =g(\infty))$  and $\sigma a\, (=g(0))$. Moreover $g$ has $1$ totally ramified value $\sigma b\, (=g(\pm \sqrt{2}\mathrm{i}))$ of order $2$. Hence $\nu_{g}= 1+1+(1 - 1/2) =5/2$. 
\end{proof}

Theorem \ref{M-thm-1} shows that \eqref{eq-KKM} in Theorem \ref{thm-KKM} is also sharp in the case $(G, k, d)=(0, 4, 4)$ because we have $\nu_{g}= 5/2$ and $2+(2/R) = 2+ \{2(0-1+2)/4\} =5/2$.  

However, we do not know that \eqref{eq-KKM} in Theorem \ref{thm-KKM} is sharp for every $(G, k, d)$. In \cite{KKM2008}, the following problem is given. 

\begin{problem}\label{con-KKM}
For every nonflat complete minimal surface of finite total curvature in $\mathbf{R}^3$, what is the best upper bound $\kappa$ of $\nu_{g}$?  
\end{problem}

By Theorems \ref{thm-KKM}, \ref{M-thm-0}, \ref{M-thm-1}, we have $5/2 \leq \kappa < 4$.   
If the genus $G$ is equal to $0$, we obtain the following more precise estimate.  

\begin{theorem}\label{thm-rami3}
Consider a nonflat complete minimal surface $X\colon \Sigma \to \mathbf{R}^3$ of genus $0$ of finite total curvature and its Gauss map $g\colon \Sigma \to \overline{\mathbf{C}}$. 
Then the number $D_{g}$ of omitted 
values and the total weight $\nu_{g}$ of the totally ramified values of $g$ satisfy 
\begin{equation}\label{eq-rami3}
D_{g} \leq \nu_{g} < 3. 
\end{equation}
In particular, we have $D_{g}\leq 2$. We thereby have $5/2\leq \kappa <3$ in the case of genus $0$ for Problem \ref{con-KKM}.  
\end{theorem}
\begin{proof}
By a suitable rotation of the surface in $\mathbf{R}^3$, we may assume that $g$ has neither zero nor pole at $p_{j}$ and that the poles of $g$ are simple. 
The simple poles of $g$ coincide with the double zeros of $hdz$ because the surface satisfies the regularity condition. By the completeness of the surface, $hdz$ has a pole at each $p_{j}$ (\cite{Ma1963}, \cite[Lemma 9.6]{Os1986}, \cite{UY2011}). Since the surface also satisfies the period condition,  
$hdz$ has a pole of order ${\mu}_{j}\geq 2$ at $p_{j}$ (\cite{Os1964}). 
Applying the Riemann-Roch theorem to the meromorphic differential $hdz$ on $\overline{\Sigma}_{G}$, 
we obtain 
\[
\displaystyle 2d-\sum_{j=1}^{k}{\mu}_{j} = 2G -2.  
\]
Thus, for $G=0$, we have
\[
\displaystyle d= 0 -1+\dfrac{1}{2}\sum_{j=1}^{k}{\mu}_{j}\geq k-1,  
\]
that is, $(k-1)/d\leq 1$. By Theorem \ref{thm-KKM}, we have 
\[
D_{g} \leq \nu_{g} \leq 2+\dfrac{2}{R} = 2+\dfrac{k-2}{d} = 2+\dfrac{k-1}{d} -\dfrac{1}{d} \leq 3-\dfrac{1}{d} <3
\]
because $d\geq 1$ holds for a nonflat complete minimal surface of finite total curvature.  
\end{proof}

\section{The Gauss images of complete minimal surfaces \\ of finite total curvature in $\mathbf{R}^4$}\label{sec-R4}

We review some basic facts about complete minimal surfaces in $\mathbf{R}^4$. For more details, 
see \cite{Ch1982, Ch1965, Fu1993, HK2018, HO1980, HO1985, Os1964}. 
Let $X=(x^1, x^2, x^3, x^4)\colon \Sigma \to \mathbf{R}^4$ be an oriented minimal surface.  
By associating a local complex coordinate $z=u+\mathrm{i}v$ with each positive isothermal coordinate $(u, v)$, $\Sigma$ is considered as a Riemann surface whose conformal 
metric is the induced metric $ds^2$ from $\mathbf{R}^4$. Then \eqref{eq-Lap} holds, that is, each $x^i\, (i=1, 2, 3, 4)$ is harmonic. 
With respect to the local complex coordinate $z=u+\mathrm{i}v$ of the surface, \eqref{eq-Lap} is given by 
\eqref{eq-partial}. Thus each $\phi_{i}:= \partial x^{i}\,dz\, (i=1, 2, 3, 4)$ is a holomorphic differential on $\Sigma$. 
These satisfy the conformality condition, the regularity condition and the period condition as in the case of $\mathbf{R}^3$. We recover $X\colon \Sigma \to \mathbf{R}^4$ by  
\begin{equation}\label{eq-integral4}
X(z) =\displaystyle \mathrm{Re} \left( \int^{z}_{z_{0}} 2\phi_{1}, \int^{z}_{z_{0}} 2\phi_{2}, \int^{z}_{z_{0}} 2\phi_{3}, \int^{z}_{z_{0}} 2\phi_{4}  \right)
\end{equation}
up to translation. Here $z_{0}$ is a fixed point of $\Sigma$. If we set 
\begin{equation}\label{eq-W-data4}
g_{1}= \dfrac{\phi_{3}+\mathrm{i}\phi_{4}}{\phi_{1}-\mathrm{i}\phi_{2}}, \quad g_{2}= \dfrac{-\phi_{3}+\mathrm{i}\phi_{4}}{\phi_{1}-\mathrm{i}\phi_{2}},\quad hdz = \phi_{1} -\mathrm{i}\phi_{2}, 
\end{equation}
then $g_{1}$ and $g_{2}$ are meromorphic functions on $\Sigma$ and $hdz$ is a holomorphic differential on $\Sigma$.
Moreover the holomorphic map $g:= (g_{1}, g_{2}) \colon \Sigma \to \overline{\mathbf{C}}\times \overline{\mathbf{C}}$ 
coincides with the Gauss map of $X(\Sigma)$. We remark that the Gauss map of $X(\Sigma)$ in $\mathbf{R}^4$ is the map from each point of $\Sigma$ to its oriented tangent plane, the set of all oriented (tangent) planes in $\mathbf{R}^{4}$ is naturally identified with the quadric 
\[
\mathbf{Q}^{2} =\{[w^{1}: w^{2}: w^{3}: w^{4}] \in \mathbf{C}\mathbf{P}^{3} \, ;\, (w^{1})^{2}+\cdots +(w^{4})^{2} = 0\}
\]
in $\mathbf{C}\mathbf{P}^{3}$ and $\mathbf{Q}^{2}$ is biholomorphic to the product of the Riemann spheres $\overline{\mathbf{C}}\times \overline{\mathbf{C}}$. For the meromorphic functions 
$g_{1}, g_{2}$ and the holomorphic differential $hdz$ given by \eqref{eq-W-data4}, we have 
\begin{equation}\label{eq-phi4}
\phi_{1}= \dfrac{1}{2}(1+g_{1}g_{2})\,hdz, \,\, \phi_{2}=\dfrac{\mathrm{i}}{2}(1-g_{1}g_{2})\,hdz, \,\,
\phi_{3}= \dfrac{1}{2}(g_{1}-g_{2})\,hdz, \,\, \phi_{4}= -\dfrac{\mathrm{i}}{2}(g_{1}+g_{2})\,hdz. 
\end{equation}  
We call $(g_{1}, g_{2}, hdz)$ the {\it Weierstrass data} ({\it W-data}, for short). Conversely, if we are given the W-data $(g_{1}, g_{2}, hdz)$ on $\Sigma$, we obtain $\phi_{1}, \phi_{2}, \phi_{3}, \phi_{4}$ by \eqref{eq-phi4}. They satisfy the conformality condition automatically, and the regularity condition is interpreted as $(hdz)_{0} = (g_{1})_{\infty} + (g_{2})_{\infty}$, where $(hdz)_{0}$ is the zero divisor of $hdz$ and 
$(g_{i})_{\infty}$ is the polar divisor of $g_{i}$ ($i=1, 2$). This is because the induced metric on 
$\Sigma$ is given by 
\begin{equation}\label{eq-metric4}
ds^2 =|h|^2(1+|g_{1}|^2)(1+|g_{2}|^2)|dz|^2. 
\end{equation} 
In a minimal surface in $\mathbf{R}^4$, it is also not so hard to find a holomorphic differential $hdz$ satisfying the regularity condition for given two meromorphic functions $g_{1}$ and $g_{2}$ on $\Sigma$, but the period condition always causes trouble. In addition, a minimal surface in $\mathbf{R}^4$ is said to be {\it complete} if all divergent paths have infinite length with respect to the metric given by \eqref{eq-metric4}. 

The Gauss curvature of $X(\Sigma)$ is given by 
\begin{equation}\label{eq-G-curvature4}
K = -\dfrac{2}{(1+|g_{1}|^2)(1+|g_{2}|^2)|h|^2} \left( \dfrac{|g'_{1}|^2}{(1+|g_{1}|^2)^2} + \dfrac{|g'_{2}|^2}{(1+|g_{2}|^2)^2}  \right), 
\end{equation}
where $g'_{1}=dg_{1}/dz, g'_{2}=dg_{2}/dz$. Moreover the total curvature of $X(\Sigma)$ is given by 
\begin{equation}\label{eq-t-curvature4}
C (\Sigma) := \displaystyle \int_{\Sigma} K\,dA 
= -\int_{\Sigma} \left( \dfrac{2|g'_{1}|^2}{(1+|g_{1}|^2)^2} + \dfrac{2|g'_{2}|^2}{(1+|g_{2}|^2)^2}  \right) du\wedge dv, \quad z=u+\mathrm{i}v, 
\end{equation}     
where $dA$ is the area element with respect to the metric \eqref{eq-metric4}. 

Fujimoto also proved the following effective estimate for the total weight of the totally ramified values of 
the Gauss map of complete minimal surfaces in $\mathbf{R}^4$. 

\begin{fact}\cite[Theorem I\hspace{-1.2pt}I\hspace{-1.2pt}I]{Fu1989}\label{fact-4-Fujimoto}
Consider a nonflat complete minimal surface $X\colon \Sigma \to \mathbf{R}^4$ and its 
Gauss map $g=(g_{1}, g_{2})$. 
Let $\nu_{g_{i}}$ be the total weight of the totally ramified values of $g_{i}\, (i=1, 2)$. 
\begin{enumerate}
\item[$\rm(\hspace{.18em}i\hspace{.18em})$] Assume that $g_{1}$ and $g_{2}$ are both nonconstant. If $\nu_{g_{1}}> 2$ 
and $\nu_{g_{2}}> 2$, then we have 
\[
\dfrac{1}{\nu_{g_{1}}-2}+\dfrac{1}{\nu_{g_{2}}-2} \geq 1. 
\] 
\item[$\rm(\hspace{.08em}ii\hspace{.08em})$] If either $g_{1}$ or $g_{2}$, say $g_{2}$, is constant, then we have
\[
\nu_{g_{1}} \leq 3. 
\]
\end{enumerate} 	
\end{fact}

Fujimoto also obtained the same estimate \cite[Theorem I\hspace{-1.2pt}I]{Fu1988} for the number of omitted values of the Gauss map of complete minimal surfaces in $\mathbf{R}^4$. These results are proved by using complex analytic methods. Examples 4.4 and 4.5 in \cite{MO1990} show 
that Fact \ref{fact-4-Fujimoto} is optimal. For a geometric interpretation of these estimates, see \cite{AAIK2017}. 

We next study the Gauss images of complete minimal surfaces of finite total curvature in $\mathbf{R}^4$. 
We also obtained the following result (\cite{HO1980, Hu1957, Os1964}).  

\begin{fact}[Huber-Osserman]\label{thm-HO4}
A complete minimal surface $X\colon \Sigma \to \mathbf{R}^4$ of finite total curvature 
satisfies: 
\begin{enumerate}
\item[$\rm(\hspace{.18em}i\hspace{.18em})$] $\Sigma$ is conformally equivalent to $\overline{\Sigma}_{G}\backslash \{ p_{1}, \ldots, p_{k} \}$, where $\overline{\Sigma}_{G}$ is a closed Riemann surface of genus $G$ and $p_{1}, \ldots, p_{k} \in \overline{\Sigma}_{G}$. Then we also call the number $G$ the genus of $X(\Sigma)$.  
\item[$\rm(\hspace{.18em}ii\hspace{.18em})$] The Weierstrass data $(g_{1}, g_{2}, hdz)$ can be extended meromorphically to $\overline{\Sigma}_{G}$.  In particular, if $d_{i}$ is the degree of 
$g_{i}\colon \overline{\Sigma}_{G} \to \overline{\mathbf{C}}$, $C(\Sigma) =-2\pi (d_{1}+d_{2})$ holds. 
\end{enumerate}
\end{fact}

The first author obtained the following estimate by geometric quantities for the total weight of the totally ramified values of the Gauss map of complete minimal surfaces of finite total curvature in $\mathbf{R}^4$. 

\begin{theorem}\cite[Theorem 3.2]{Ka2009}\label{R4-fact-1}
Let $X\colon \Sigma = \overline{\Sigma}_{G} \backslash \{ p_{1}, \ldots, p_{k} \} \to \mathbf{R}^4$ be a nonflat complete minimal surface of finite total curvature, $g=(g_{1}, g_{2})\colon \Sigma \to \overline{\mathbf{C}}\times \overline{\mathbf{C}}$ be its Gauss map, $d_{i}$ be the degree of $g_{i}\colon \overline{\Sigma}_{G} \to \overline{\mathbf{C}}$ and 
$\nu_{g_{i}}$ be the total weight of the totally ramified values of $g_{i}$ ($i=1, 2$).  
\begin{enumerate}
\item[$\rm(\hspace{.18em}i\hspace{.18em})$] Assume that $g_{1}$ and $g_{2}$ are both nonconstant. 
If $\nu_{g_{1}}> 2$ and $\nu_{g_{2}}> 2$, then we have
\begin{equation}\label{eq-4-Kaw-1}
\dfrac{1}{\nu_{g_{1}}-2} + \dfrac{1}{\nu_{g_{2}}-2} \geq R_{1}+R_{2}, \quad R_{i} =\dfrac{d_{i}}{2G -2+k}\, (i=1, 2),  
\end{equation}
and $R_{1}+R_{2}> 1$ holds. We thus obtain  
\begin{equation}\label{eq-4-Kaw-2}
\dfrac{1}{\nu_{g_{1}}-2} + \dfrac{1}{\nu_{g_{2}}-2} >1. 
\end{equation} 
\item[$\rm(\hspace{.08em}ii\hspace{.08em})$] If either $g_{1}$ or $g_{2}$, say $g_{2}$, is constant, 
then we have 
\begin{equation}\label{eq-4-Kaw-3}
\nu_{g_{1}} \leq 2+\dfrac{1}{R_{1}}, \quad \dfrac{1}{R_{1}} = \dfrac{2G -2+k}{d_{1}},   
\end{equation}
and $1/R_{1}<1$. We thus obtain $\nu_{g_{1}}< 3$. 
\end{enumerate}   
\end{theorem}

Theorem \ref{R4-fact-1} is proved by using the methods of complex algebraic geometry. We obtain the following result as an immediate consequence of Theorem \ref{R4-fact-1}.  

\begin{corollary}\cite[Theorem 6.10]{HO1980}\label{R4-cor-2} 
Let $X\colon \Sigma \to \mathbf{R}^4$ be a complete minimal surface of finite total curvature and $g=(g_{1}, g_{2})\colon \Sigma \to \overline{\mathbf{C}}\times \overline{\mathbf{C}}$ be its Gauss map. 
\begin{enumerate}
\item[$\rm(\hspace{.18em}i\hspace{.18em})$] If both $g_{1}$ and $g_{2}$ omit more than $3$ values, then $X(\Sigma)$ must be a plane. 
\item[$\rm(\hspace{.08em}ii\hspace{.08em})$] If either $g_{1}$ or $g_{2}$, say $g_{2}$, is constant and $g_{1}$ omits more than $2$ values, then $X(\Sigma)$ must be a plane.  
\end{enumerate}
\end{corollary}

\begin{example}[Lagrangian catenoid in $\mathbf{C}^2$]\label{R4-ex-1}
Consider $\Sigma =\overline{\mathbf{C}}\backslash \{ 0, \infty \}$ and 
\begin{equation}\label{R4-w-data1}
(g_{1}(z), g_{2}(z), hdz) = \left( -z^2, c, -\dfrac{dz}{z^2} \right), \quad c\in \mathbf{C}. 
\end{equation}
Then we easily show that this Weierstrass data satisfies the conformality, regularity and period conditions. 
Thus we obtain a complete minimal surface in $\mathbf{R}^4$ whose W-data is given by \eqref{R4-w-data1}. 
In particular, the surface has finite total curvature and $g_{1}$ has $2$ omitted values $0, \infty$. Thus $D_{g_{1}}=\nu_{g_{1}}=2$ holds and the surface shows that 
$\rm(\hspace{.08em}ii\hspace{.08em})$ in Corollary \ref{R4-cor-2} is optimal. Moreover the surface also shows 
that \eqref{eq-4-Kaw-3} in Theorem \ref{R4-fact-1} is sharp 
in the case $(G, k, d_{1})= (0, 2, 2)$ because we have $\nu_{g_{1}}=2$ and $2+(1/R_{1}) = 2+ \{ (0-2+2)/2 \}=2$. 
This surface is called the Lagrangian catenoid in $\mathbf{C}^2 (\simeq \mathbf{R}^4)$ (see \cite{CF1999}).    
\end{example}

We obtain the following example which shows that \eqref{eq-4-Kaw-1} is optimal under some geometric condition.   

\begin{theorem}\label{thm-R4-HKW-1}
Consider $\Sigma =\overline{\mathbf{C}}\backslash \{ \pm \mathrm{i}, \infty \}$ and 
\begin{equation}\label{R4-w-data2}
(g_{1}(z), g_{2}(z), hdz) = \left( \dfrac{z^2 +a}{z^2 -1}, \dfrac{z^2 +b}{z^2 -1}, \dfrac{(z^2 -1)^2}{(z^2 +1)^2}\,dz \right).
\end{equation} 
For any $a, b\in \mathbf{R}$ satisfying $(a+1)(b+1)=8$, we obtain a complete minimal surface $X\colon \Sigma \to \mathbf{R}^4$ whose Weierstrass data is given by \eqref{R4-w-data2}. 
In particular, the surface $X(\Sigma)$ has finite total curvature and $\nu_{g_{1}}=\nu_{g_{2}}=5/2$. 
\end{theorem}

\begin{proof}
We first prove that a minimal surface $X\colon \Sigma \to \mathbf{R}^4$ whose W-data is given by \eqref{R4-w-data2} can be constructed. We can easily verify that \eqref{R4-w-data2} satisfies the conformality and regularity conditions. So it is sufficient to prove that \eqref{R4-w-data2} satisfies the period condition. We denote small 
circles about $z= \mathrm{i}, -\mathrm{i}$ by $c_{\mathrm{i}}, c_{-\mathrm{i}}$, respectively. 
Then we obtain 
\[
\begin{aligned}
& \int_{c_{\mathrm{i}}} \phi_{1} = - \int_{c_{-\mathrm{i}}} \phi_{1} = \dfrac{\pi}{4}(ab+a+b-7) = 0, \\
& \int_{c_{\mathrm{i}}} \phi_{2} = - \int_{c_{-\mathrm{i}}} \phi_{2} = -\dfrac{\pi}{4}\mathrm{i}(ab+a+b+1) = -2\pi\mathrm{i}, \\
& \int_{c_{\mathrm{i}}} \phi_{3} = - \int_{c_{-\mathrm{i}}} \phi_{3} = 0,  \quad \int_{c_{\mathrm{i}}} \phi_{4} = - \int_{c_{-\mathrm{i}}} \phi_{4} = 2\pi\mathrm{i}. 
\end{aligned}
\]
Since $H_{1}(\Sigma, \mathbf{Z}) = \mathbf{Z} \{ c_{\mathrm{i}}, c_{-\mathrm{i}} \}$, each $\phi_{i}\, (i=1, 2, 3, 4)$ has no real periods on $\Sigma$, 
that is, \eqref{R4-w-data2} satisfies the period condition. 

Since the line element $ds:= |h|\sqrt{(1+|g_{1}|^2)(1+|g_{2}|^2)}|dz|$ of $X(\Sigma)$ can be represented as 
\[
\dfrac{|z^2 -1|^2}{|z^2 +1|^2}\sqrt{\left( 1+ \dfrac{|z^2 +a|^2}{|z^2 -1|^2} \right)\left( 1+ \dfrac{|z^2 +b|^2}{|z^2 -1|^2} \right)} |dz|,  
\]
we see that $\int_{\gamma} ds=+\infty$ for any divergent path $\gamma$ on $\Sigma$. Hence $X(\Sigma)$ is complete. 
We conclude that $X(\Sigma)$ is a complete minimal surface 
of finite total curvature in $\mathbf{R}^4$ and has $C(\Sigma) = -2\pi (2+2) = -8\pi$.  

We next show $\nu_{g_{1}}= \nu_{g_{2}} =5/2$. The function $g_{1}$ has $2$ omitted values $(1-a)/2\, (= g_{1}(\pm \mathrm{i}))$, $1\, (=g_{1}(\infty))$ and $1$ totally ramified value $-a\, (=g_{1}(0))$ of order $2$. Hence $\nu_{g_{1}} =1+1+(1-1/2)=5/2$. 
In the same way, the function $g_{2}$ has $2$ omitted values $(1-b)/2\, (= g_{2}(\pm \mathrm{i}))$,  
$1\, (=g_{2}(\infty))$ and $1$ totally ramified value $-b\, (=g_{2}(0))$ of order $2$. Hence $\nu_{g_{2}} =1+1+(1-1/2)=5/2$. 
\end{proof}

\begin{remark}\label{rmk-R4-HKW-1}
In Theorem \ref{thm-R4-HKW-1}, when $a=b$, $X(\Sigma)$ is then $1$-decomposable (See \cite[\S 4]{HO1980} for the definition of $h$-decomposable).  
This is because we have $\phi_{3}\equiv 0$. 
\end{remark}

Theorem \ref{thm-R4-HKW-1} shows that \eqref{eq-4-Kaw-1} in Theorem \ref{R4-fact-1} is optimal 
in the case $(G, k, d_{1}, d_{2}) = (0, 3, 2, 2)$ because we have 
\[
\dfrac{1}{\nu_{g_{1}} -2} + \dfrac{1}{\nu_{g_{2}} -2} =2 +2 =4  
\] 
and $R_{1}+R_{2} = 2/(0-2+3) +2/(0-2+3)= 4$. 

However, we do not know that \eqref{eq-4-Kaw-1} in Theorem \ref{R4-fact-1} is optimal for every $(G, k, d_{1}, d_{2})$ and $\rm(\hspace{.18em}i\hspace{.18em})$ of  
Corollary \ref{R4-cor-2} is sharp. As in the case of $\mathbf{R}^3$, the following problem can be considered. 

\begin{problem}\label{pro-R4-KKM}
For every nonflat complete minimal surface of finite total curvature in $\mathbf{R}^4$ whose nonconstant Gauss maps $g_{1}$ and $g_{2}$ have $\nu_{g_{1}}> 2$ and $\nu_{g_{2}}>2$, 
what is the best lower bound $\kappa$ of $1/(\nu_{g_{1}}-2) + 1/(\nu_{g_{2}}-2)$?
\end{problem}

By Theorems \ref{R4-fact-1} and \ref{thm-R4-HKW-1}, we have $1<\kappa \leq 4$. 
If the genus $G$ is equal to $0$, we show the following more precise estimate.    

\begin{theorem}\label{thm-R4-M2}
Consider a nonflat complete minimal surface $X\colon \Sigma \to \mathbf{R}^4$ of genus $0$ of finite 
total curvature and its Gauss map $g=(g_{1}, g_{2})\colon \Sigma \to \overline{\mathbf{C}}\times \overline{\mathbf{C}}$.  
Let $\nu_{g_{i}}$ be the total weight of the totally ramified values of $g_{i}\,(i=1, 2)$. 
Assume that both $g_{1}$ and $g_{2}$ are nonconstant. If $\nu_{g_{1}} > 2$ and $\nu_{g_{2}} > 2$, then we have 
\begin{equation}\label{eq-R4-M2}
\dfrac{1}{\nu_{g_{1}}-2} + \dfrac{1}{\nu_{g_{2}}-2} >2. 
\end{equation}
We thereby have $2<\kappa \leq 4$ in the case of genus $0$ for Problem \ref{pro-R4-KKM}.  
\end{theorem}

\begin{proof}
By a suitable isometric transformation of the surface in $\mathbf{R}^4$, 
we assume that both $g_{1}$ and $g_{2}$ are no pole at $p_{j}$ and have only simple poles. 
By the completeness of the surface, $hdz$ has a pole at $p_{j}$ (\cite{Ma1963}, \cite[Lemma 9.6]{Os1986}, \cite{UY2011}). 
Since the surface also satisfies the period condition,  $hdz$ has a pole of order ${\mu}_{j}\geq 2$ at $p_{j}$ (\cite{Os1964}). 
Moreover the total order of zeros of $hdz$ on $\overline{\Sigma}_{G}$ is $d_{1}+d_{2}$ due to the regularity condition. 
Applying the Riemann-Roch theorem to the meromorphic differential $hdz$ on $\overline{\Sigma}_{G}$, 
we obtain 
\begin{equation}\label{eq-R4-M2-1}
d_{1}+d_{2} -\displaystyle \sum_{j=1}^{k} \mu_{j} = 2G-2.  
\end{equation}
We consider $G=0$. We first assume $k\leq 2$. Then $\Sigma$ is biholomorphic to $\mathbf{C}\,\, (k=1)$  or $\mathbf{C}\backslash \{ 0 \}\,\, (k=2)$. 
If $k=1$, then $\nu_{g_{1}}< 2$ and $\nu_{g_{2}}< 2$ from Proposition \ref{prop-rational}. 
For the case of $k=2$, we have $\nu_{g_{1}}\leq 2$ and $\nu_{g_{2}}\leq 2$ from Remark \ref{rmk-puncture}. 
We thus may assume $k\geq 3$.  
Since $\mu_{j}\geq 2$ holds for each $j$, \eqref{eq-R4-M2-1} implies that 
\[
d_{1}+d_{2} = \displaystyle \sum_{j=1}^{k} \mu_{j} -2 \geq 2(k-1), \quad \mathrm{i.e.,} \quad \dfrac{d_{1}+d_{2}}{k-1} \geq 2. 
\]
By \eqref{eq-4-Kaw-1} in Theorem \ref{R4-fact-1}, we have 
\[
\dfrac{1}{\nu_{g_{1}}-2} + \dfrac{1}{\nu_{g_{2}}-2} \geq R_{1}+R_{2} = \dfrac{d_{1}+d_{2}}{k-2} > \dfrac{d_{1}+d_{2}}{k-1} \geq 2. 
\]
\end{proof}

As an application of Theorem \ref{thm-R4-M2}, we can give more refined estimate than Corollary \ref{R4-cor-2} when the genus of the surface is $0$.  

\begin{corollary}\label{thm-R4-M3}
Let $X\colon \Sigma \to \mathbf{R}^4$ be a nonflat complete minimal surface of genus $0$ of finite 
total curvature and $g=(g_{1}, g_{2})\colon \Sigma \to \overline{\mathbf{C}}\times \overline{\mathbf{C}}$ be 
its Gauss map. If both $g_{1}$ and $g_{2}$ are nonconstant, then $g_{1}$ or $g_{2}$ can omit at most $2$ values in $\overline{\mathbf{C}}$. 
\end{corollary}

\begin{proof}
Let $D_{g_{i}}$ be the number of omitted values of $g_{i}\,(i=1, 2)$. It is proved by contradiction.  
Suppose that $D_{g_{1}} \geq 3$ and $D_{g_{2}} \geq 3$. Then we have  
\begin{equation}\label{eq-M-cor-2}
\dfrac{1}{D_{g_{1}}-2} + \dfrac{1}{D_{g_{2}}-2} \leq \dfrac{1}{3-2} + \dfrac{1}{3-2} =2. 
\end{equation}
On the other hand, by $D_{g_{i}} \leq \nu_{g_{i}}$ and Theorem \ref{thm-R4-M2}, we obtain  
\begin{equation}
\dfrac{1}{D_{g_{1}}-2} + \dfrac{1}{D_{g_{2}}-2} \geq  \dfrac{1}{\nu_{g_{1}}-2} + \dfrac{1}{\nu_{g_{2}}-2}>2. 
\end{equation}
It contradicts \eqref{eq-M-cor-2}. 
\end{proof}

\begin{remark}\label{rmk-R4-M3}
By Corollary \ref{thm-R4-M3}, we see that for the case of genus $0$, the inequality corresponding to \eqref{eq-4-Kaw-1} in Theorem \ref{R4-fact-1} does not arise in the estimate for the number of omitted values. Theorem \ref{thm-R4-HKW-1} is an example to show that  \eqref{eq-4-Kaw-1} in Theorem \ref{R4-fact-1} can occur for the total weight of the totally ramified values of the Gauss maps of complete minimal surfaces of genus $0$ of finite total curvature in $\mathbf{R}^4$. 
\end{remark}

Finally, we give complete minimal surfaces of genus $0$ of finite total curvature in $\mathbf{R}^4$ 
whose both $g_{1}$ and $g_{2}$ omit $2$ values. 
These examples show that Corollary \ref{thm-R4-M3} is sharp. 
The first example (Theorem \ref{pro-R4-M4}) is an extension of the 
W-data of \cite[Theorem 1]{MS1994}.     

\begin{theorem}\label{pro-R4-M4}
Consider $\Sigma =\overline{\mathbf{C}} \backslash \{ 0, \infty \}$ and 
\begin{equation}\label{eq-pro-R4-M4-1}
(g_{1}(z), g_{2}(z), hdz) = \left( \dfrac{z^m -a}{z^m- 1}, \dfrac{z^n -b}{z^n -1}, \dfrac{(z^m -1)(z^n -1)}{z^l}\,dz \right), 
\end{equation}
where $a, b \in \mathbf{R} \backslash \{ 0, 1\}$ and $l, m, n$ are positive integers satisfying 
\[
m+n >2l-2\geq 2, \quad m-l\not= -1, \quad n-l\not= -1. 
\]
Then we obtain a complete minimal surface $X\colon \Sigma \to \mathbf{R}^4$ whose Weierstrass data is 
given by \eqref{eq-pro-R4-M4-1}. In particular, $X(\Sigma)$ has finite total curvature, $g_{1}$ omits two values $1, a$ and $g_{2}$ omits two values $1, b$.  
\end{theorem}

\begin{proof}
Since we check at once that $g_{1}$ omits two values $1\,(=g_{1}(\infty)),\, a\,(=g_{1}(0))$ and $g_{2}$ omits two values $1\,(=g_{2}(\infty)),\, b\,(=g_{2}(0))$,  
we here prove that a complete minimal surface $X\colon \Sigma \to \mathbf{R}^4$ of finite 
total curvature whose W-data is given by \eqref{eq-pro-R4-M4-1} can be constructed. 
We can easily verify that \eqref{eq-pro-R4-M4-1} satisfies the conformality and regularity conditions. 
We show that \eqref{eq-pro-R4-M4-1} satisfies the period condition.  
Here $c_{0}$ stands for a small circle about $z=0$. Then we have 
\[
\begin{aligned}
& \int_{c_{0}} \phi_{1} = \dfrac{1}{2} \int_{c_{0}} \{ 2z^{m+n-l} -(b+1)z^{m-l} -(a+1)z^{n-l} +(ab+1)z^{-l} \}\,dz, \\
& \int_{c_{0}} \phi_{2} = \dfrac{\mathrm{i}}{2} \int_{c_{0}} \{ (b-1)z^{m-l} +(a-1)z^{n-l} +(1-ab)z^{-l} \}\,dz, \\
& \int_{c_{0}} \phi_{3} = \dfrac{1}{2} \int_{c_{0}} \{ (b-1)z^{m-l} -(a-1)z^{n-l} +(a-b)z^{-l} \}\,dz, \\
& \int_{c_{0}} \phi_{4} = -\dfrac{\mathrm{i}}{2} \int_{c_{0}} \{ 2z^{m+n-l} -(b+1)z^{m-l} -(a+1)z^{n-l} +(a+b)z^{-l} \}\,dz. 
\end{aligned}
\]
By $m+n> 2l-2$, we have $m+n-l > l-2 \geq 1-2 = -1$, that is, $m+n-l \geq 0$. Moreover, by assumption, 
$m-l\not= -1, \, n-l\not= -1$ and $-l\leq -2$ hold. Then we have 
\[   
\int_{c_{0}} \phi_{1} = \int_{c_{0}} \phi_{2} = \int_{c_{0}} \phi_{3} = \int_{c_{0}} \phi_{4} = 0.    
\]
Since $H_{1}(\Sigma, \mathbf{Z}) = \mathbf{Z} c_{0} $, \eqref{eq-pro-R4-M4-1} satisfies the period condition. 
We next prove that the surface $X(\Sigma)$ is complete.  Since the line element of $X(\Sigma)$ can be represented as 
\[
ds= \dfrac{1}{|z|^{l}}\sqrt{(|z^m -1|^2 +|z^m -a|^2)(|z^n -1|^2 +|z^n -b|^2)}\,|dz|, 
\]
we see that $\int_{\gamma} ds = +\infty$ for any divergent path $\gamma$ on $\Sigma$. 
Hence $X(\Sigma)$ is a complete minimal surface of finite total curvature in $\mathbf{R}^4$ and has $C(\Sigma) =-2\pi (m+n)$. 
\end{proof}

\begin{theorem}\label{pro-R4-M5}
Consider $\Sigma =\overline{\mathbf{C}} \backslash \{ 0, \infty \}$ and 
\begin{equation}\label{eq-pro-R4-M5-1}
(g_{1}(z), g_{2}(z), hdz) = \left( az, -\bar{a}z, \dfrac{dz}{z^2} \right), 
\end{equation}
where $a \in \mathbf{C} \backslash \{ 0 \}$. Then we obtain a complete minimal surface 
$X\colon \Sigma \to \mathbf{R}^4$ whose Weierstrass data is 
given by \eqref{eq-pro-R4-M5-1}. In particular, $X(\Sigma)$ has finite total curvature, both $g_{1}$ and $g_{2}$ omit two values $0, \infty$.   
\end{theorem}

\begin{proof}
Since we check at once that both $g_{1}$ and $g_{2}$ omit two values $0\,(=g_{i}(0)),\, \infty\,(=g_{i}(\infty))$, we here prove that a complete minimal surface $X\colon \Sigma \to \mathbf{R}^4$ of finite 
total curvature whose W-data is given by \eqref{eq-pro-R4-M5-1} can be constructed. 
We can easily verify that \eqref{eq-pro-R4-M5-1} satisfies the conformality and regularity conditions. 
We show that \eqref{eq-pro-R4-M5-1} satisfies the period condition.  
Here $c_{0}$ stands for a small circle about $z=0$. Then we have 
\[
\begin{aligned}
& \int_{c_{0}} \phi_{1} = \dfrac{1}{2} \int_{c_{0}} \left( \dfrac{1}{z^2} - |a|^2 \right)\,dz =0, \\
& \int_{c_{0}} \phi_{2} = \dfrac{\mathrm{i}}{2} \int_{c_{0}} \left( \dfrac{1}{z^2} + |a|^2 \right) \,dz =0, \\
& \int_{c_{0}} \phi_{3} = \dfrac{a+\bar{a}}{2} \int_{c_{0}} \dfrac{dz}{z}\,dz = 2\pi\mathrm{i}\mathrm{Re}(a)\in \mathrm{i}\mathbf{R}, \\
& \int_{c_{0}} \phi_{4} = -\mathrm{i}\dfrac{a-\bar{a}}{2} \int_{c_{0}} \dfrac{dz}{z} =  2\pi\mathrm{i}\mathrm{Im}(a)\in \mathrm{i}\mathbf{R}.  
\end{aligned}
\]
Since $H_{1}(\Sigma, \mathbf{Z}) = \mathbf{Z} c_{0} $, each $\phi_{i}\, (i=1, 2, 3, 4)$ has no real periods on $\Sigma$, 
that is, \eqref{eq-pro-R4-M5-1} satisfies the period condition. 
We next prove that $X(\Sigma)$ is complete. Since the line element of $X(\Sigma)$ can be represented as 
\[
ds= \dfrac{1+|a|^2|z|^2}{|z|^2}\,|dz|, 
\]
we see that $\int_{\gamma} ds = +\infty$ for any divergent path $\gamma$ on $\Sigma$. 
Hence $X(\Sigma)$ is a complete minimal surface of finite total curvature in $\mathbf{R}^4$ and has $C(\Sigma) = -2\pi (1+1) = -4\pi$. 
\end{proof}

\begin{remark}\label{rmk-R4-M5}
In Theorem \ref{pro-R4-M5}, if $a+\bar{a} =0$, that is, $a\in \mathrm{i}\mathbf{R}$ holds, then 
$X(\Sigma)$ is $1$-decomposable because we have $\phi_{3} \equiv 0$. 
Moreover, if $a-\bar{a} =0$, that is, $a\in \mathbf{R}$ holds, then 
$X(\Sigma)$ is also $1$-decomposable because we have $\phi_{4} \equiv 0$. 
\end{remark}

\section{Outstanding problem and conjecture}\label{sec-Prob}
In this section, we discuss the following outstanding problem and conjecture in the study of value distribution of the Gauss map of  complete minimal surfaces in Euclidean space.  

\subsection{The generalized Gauss images of complete nonorientable minimal surfaces of finite total curvature} 

We summarize some basic facts of nonorientable minimal surfaces in ${\R}^{3}$. 
For more details, we refer the reader to \cite{AFL2020, AFL2021, El1986, LM2000, Ma2005, Mi1981}. 
Let $\widehat{X}\colon \widehat{\Sigma}\to {\mathbf{R}}^{3}$ be a conformal minimal immersion of a nonorientable 
Riemann surface $\widehat{\Sigma}$ in ${\mathbf{R}}^{3}$. If we consider the orientable conformal double 
cover $\pi \colon \Sigma \to \widehat{\Sigma}$, then the composition $X:=\widehat{X}\circ \pi \colon \Sigma \to {\mathbf{R}}^{3}$ 
is a conformal minimal immersion of the orientable Riemann surface $\Sigma$ in ${\mathbf{R}}^{3}$. 
Let $I\colon \Sigma \to \Sigma$ denote the antiholomorphic order two deck transformation associated to the orientable 
cover $\pi \colon \Sigma \to \widehat{\Sigma}$, then $I^{\ast} ({\phi}_{i})=\bar{\phi}_{i}\, (i=1, 2, 3)$ or equivalently, 
\begin{equation}\label{eq-appl-nonori-1}
g\circ I = -\dfrac{1}{\bar{g}}, \quad  I^{\ast}(hdz) = -g^2 hdz.  
\end{equation}
Conversely, if $(g, hdz)$ is the Weierstrass data of an orientable minimal surface $X (\Sigma)$ in  ${\mathbf{R}}^{3}$ and 
$I$ is an antiholomorphic involution without fixed points in $\Sigma$ satisfying (\ref{eq-appl-nonori-1}), then the unique map 
$\widehat{X}\colon \widehat{\Sigma}=\Sigma /\langle I \rangle \to {\mathbf{R}}^{3}$ satisfying that $X=\widehat{X}\circ \pi$ is 
a nonorientable minimal surface in ${\mathbf{R}}^{3}$. 

The fact that $g\circ I= -(\bar{g})^{-1}$ implies the existence of a map $\hat{g}\colon \widehat{\Sigma} \to \mathbf{RP}^{2}$ 
satisfying $\hat{g}\circ \pi = {\pi}_{0} \circ g$, where ${\pi}_{0}\colon \overline{\mathbf{C}} \to \R\Pi^{2}\equiv \overline{\mathbf{C}} /\langle I_{0} \rangle$ 
is the natural projection and $I_{0}:=-(\bar{z})^{-1}$ is the antipodal map of $\overline{\mathbf{C}}$. 
We call the map 
$\hat{g}\colon \widehat{\Sigma} \to \mathbf{RP}^2$ the {\it generalized Gauss map} of $\widehat{X}(\widehat{\Sigma})$. 
Applying Fact \ref{thm-rami2} to their generalized Gauss maps, the following result holds. 

\begin{fact}\cite{LM2000}\label{fa-nonori-LM}
The generalized Gauss map of a nonflat complete nonorientable minimal surface in $\mathbf{R}^3$ can omit at most $2$ points in $\mathbf{R}\mathbf{P}^2$. 
\end{fact}

Fact \ref{fa-nonori-LM} is optimal because L\'opez and Mart\'in \cite{LM2000} showed that there exist complete nonorientable minimal surfaces in $\mathbf{R}^3$ whose generalized Gauss maps omit  
$2$ points in $\mathbf{R}\mathbf{P}^2$. Fact \ref{fact-Oss} implies that the generalized Gauss map of a nonflat complete nonorientable minimal 
surface of finite total curvature in $\mathbf{R}^3$ can omit at most $1$ points in $\mathbf{R}\mathbf{P}^2$. However,  in all known complete nonorientable minimal 
surfaces of finite total curvature in $\mathbf{R}^3$, the generalized Gauss maps are surjective. From this fact, L\'opez and Mart\'in proposed the following problem. 

\begin{problem}\cite{LM2000, Ma2005}\label{P-nonori}
Does there exist a complete nonorientable minimal surface of finite total curvature in $\mathbf{R}^3$ whose generalized Gauss map 
omits one point in $\mathbf{R}\mathbf{P}^2$? 
\end{problem}

In an analogous way, we obtained the following result for complete nonorientable minimal surfaces of finite total curvature in $\mathbf{R}^4$. 
For more details, see \cite{AAIK2017}. 

\begin{proposition}\cite[Proposition 3.7]{AAIK2017}\label{thm-appl-nonori-3} 
Let $\widehat{X}\colon \widehat{\Sigma}\to {\mathbf{R}}^{4}$ be a nonflat complete nonorientable minimal surface of finite total curvature 
and $\widehat{G}=(\hat{g_{1}}, \hat{g_{2}})$ be the generalized Gauss map of $\widehat{X}(\widehat{\Sigma})$. 
\begin{enumerate}
\item[(i)] Assume that $\hat{g}_{1}$ and $\hat{g}_{2}$ are both nonconstant. Then at least one of them can 
omit at most $1$ point in $\R\Pi^{2}$. 
\item[(i\hspace{-.1em}i)] If either $\hat{g}_{1}$ or $\hat{g}_{2}$, say $\hat{g}_{2}$, is constant, 
then $\hat{g}_{1}$ can omit at most $1$ point in $\R\Pi^{2}$. 
\end{enumerate}
\end{proposition} 

However, we also do not know whether Proposition \ref{thm-appl-nonori-3} is optimal or not. 

\subsection{Flat point conjecture} At the end of this paper, we explain the following conjecture. 

\begin{conjecture}[Flat point conjecture]\label{P-flat}
If a nonflat complete minimal surface in $\mathbf{R}^3$ has at least $1$ flat point, then its Gauss map omits at most $3$ values. 
\end{conjecture}

A {\it flat point} means a point where the Gauss curvature of the surface vanishes. Conjecture \ref{P-flat} can be rephrased as follows: if the Gauss map of a complete minimal surface in $\mathbf{R}^3$ has exactly $4$ omitted values, then the Gauss curvature is strictly negative on everywhere, that is, the surface has no flat point. Conjecture \ref{P-flat} is true if a complete minimal surface is ``pseudo-algebraic''.     
Here a complete minimal surface $X\colon \Sigma \to \mathbf{R}^3$ is said to be {\it pseudo-algebraic} if the following conditions are satisfied: 
\begin{enumerate}
\item[(i)] The W-data $(g, hdz)$ is defined on a Riemann surface $\overline{\Sigma}_{G}\backslash \{ p_{1}, \ldots, p_{k} \}$, where $\overline{\Sigma}_{G}$ is a closed Riemann surface of genus $G$ 
and $p_{1}, \ldots, p_{k}\in \overline{\Sigma}_{G}$.   
\item[(i\hspace{-.1em}i)] The W-data $(g, hdz)$ can be extended meromorphically to $\overline{\Sigma}_{G}$.  
\end{enumerate}
This class does not assume the period condition on $\Sigma$, that is, a surface in this class does not necessarily have to be well-defined on $\Sigma$. 
If a pseudo-algebraic minimal surface is not well-defined on $\Sigma$, it is defined on some covering surface of $\Sigma$, in the worst case, on the universal covering surface.  
Note that Gackstatter called such surfaces {\it abelian minimal surfaces} \cite{Ga1976}. 
Complete minimal surfaces of finite total curvature, the classical Scherk surface and the Voss surface are all pseudo-algebraic.  

The reason why Conjecture \ref{P-flat} holds for nonflat pseudo-algebraic minimal surfaces is as follows:  
If the number of (not necessarily totally) ramified values except for omitted values of $g$ is $l$, the first author, Kobayashi and Miyaoka \cite[Theorem 3.3]{KKM2008} proved that
\[
D_{g} \leq 2+\dfrac{2}{R} -\dfrac{l}{d}. 
\]
For pseudo-algebraic minimal surfaces, we have $1/R \leq 1$. We thus obtain $D_{g}\leq 4-(l/d)$. 
By  \eqref{eq-G-curvature}, a flat point of the surface coincides with a point where $g'$ vanishes. 
If the surface has at least $1$ flat point, then $l\geq 1$ holds. Hence we have $D_{g}< 4$, that is, $D_{g}\leq 3$.  

We do not know that Conjecture \ref{P-flat} is also true or not for nonflat complete minimal surfaces which are not pseudo-algebraic in $\mathbf{R}^3$. 


\end{document}